\documentclass[11pt]{article}

\usepackage{a4wide}

\usepackage{amsmath,amsbsy,amssymb,latexsym}
\usepackage{amssymb,amsthm}

\usepackage{cite}

\newtheorem{theorem}{Theorem}[section]
\newtheorem{corollary}{Corollary}

\newtheorem{definition}[theorem]{Definition}
\newtheorem{remark}{Remark}
\newtheorem{example}{Example}

\newcommand{\R}{\mathbb{R}}

\begin{document}

\title{The generalized natural boundary conditions\\
for fractional variational problems\\
in terms of the Caputo derivative\thanks{Accepted (19 February 2010) 
for publication in \emph{Computers and Mathematics with Applications}.}}

\author{Agnieszka B. Malinowska$^{1, 2}$\\
{\tt abmalinowska@ua.pt} \and
Delfim F. M. Torres$^1$\\
{\tt delfim@ua.pt}}

\date{$^1$Department of Mathematics,
University of Aveiro,
3810-193 Aveiro, Portugal\\[0.3cm]
$^2$Faculty of Computer Science,
Bia{\l}ystok University of Technology,\\
15-351 Bia\l ystok, Poland}

\maketitle


\begin{abstract}
This paper presents necessary and sufficient optimality conditions
for problems of the fractional calculus of variations with a
Lagrangian depending on the free end-points. The fractional
derivatives are defined in the sense of Caputo.

\bigskip


\smallskip

\end{abstract}


\section{Introduction}

Fractional calculus is one of the generalizations of the classical
calculus and it has been used successfully in various fields of
science and engineering --- see, \textrm{e.g.},
\cite{debnath,Diethelm,ferreira,hilfer,hilfer2,%
kulish,magin,metzler,miller,Oustaloup,TM,Zas}.
In recent years, there has been a growing interest in the area of
fractional variational calculus and its applications
\cite{agrawal2,agrawalCap,CD:Agrawal:2007,R:A:D:10,Almeida2,Atanackovic,%
Baleanu:Agrawal,Baleanu,El-Nabulsi:Torres07,R:T:08,%
Frederico:Torres07,Frederico:Torres08,G:D:10,klimek,muslih,rie}.
Applications include classical and quantum mechanics,
field theory, and optimal control. In the papers
cited above, the problems have been formulated mostly in terms of
two types of fractional derivatives, namely Riemann--Liouville and
Caputo. The natural boundary conditions for fractional variational
problems, in terms of the Riemann--Liouville and the Caputo
derivative, are presented in \cite{agrawal2,agrawalCap}.
Here we develop further the theory by proving necessary optimality
conditions for more general problems of the fractional calculus of
variations with a Lagrangian that may also depend on the unspecified
end-points $y(a)$, $y(b)$. More precisely, the problem under our
study consists to minimize a functional which is defined in terms of the
Caputo derivative and having no constraint on $y(a)$ and/or $y(b)$.
The novelty is the dependence of the integrand $L$ on the free
end-points $y(a)$, $y(b)$. This class of problems is motivated
by applications in the field of economics \cite{JP:DT:AZ}.

The paper is organized as follows. Section~\ref{sec2} presents the
necessary definitions and concepts of the fractional calculus; our
results are formulated, proved, and illustrated through an example
in Section~\ref{main}. Main results of the paper include necessary
optimality conditions with the new natural boundary conditions
(Theorem~\ref{Theo E-L1}) that become sufficient under appropriate
convexity assumptions (Theorem~\ref{suff}).


\section{Fractional Calculus}
\label{sec2}

In this section we review the necessary definitions and facts from
fractional calculus. For more on the subject we refer the reader to
\cite{kilbas,Oldham,Podlubny,samko}.

Let $f\in L_1([a,b])$  and $0<\alpha<1$. We begin with the left and
the right Riemann--Liouville Fractional Integrals (RLFI) of order
$\alpha$ of function $f$ which are defined as: the left RLFI
\begin{equation}
\label{RLFI1}
{_aI_x^\alpha}f(x)=\frac{1}{\Gamma(\alpha)}\int_a^x
(x-t)^{\alpha-1}f(t)dt,\quad x\in[a,b],
\end{equation}
the right RLFI
\begin{equation}
\label{RLFI2}
{_xI_b^\alpha}f(x)=\frac{1}{\Gamma(\alpha)}\int_x^b(t-x)^{\alpha-1}
f(t)dt,\quad x\in[a,b],
\end{equation}
where $\Gamma(\cdot)$ represents the Gamma function. Moreover,
${_aI_x^0}f={_xI_b^0}f=f$ if $f$ is a continuous function.

Let $f\in AC([a,b])$, where $AC([a,b])$ represents the space
of absolutely continuous functions on $[a,b]$. Then using equations
\eqref{RLFI1} and \eqref{RLFI2}, the left and the right
Riemann--Liouville and Caputo derivatives are defined as: the left
Riemann--Liouville Fractional Derivative (RLFD)
\begin{equation}\label{RLFD1}
{_aD_x^\alpha}f(x)=\frac{1}{\Gamma(1-\alpha)}\frac{d}{dx}\int_a^x
(x-t)^{-\alpha}f(t)dt=\frac{d}{dx}{_aI_x^{1-\alpha}}f(x),\quad
x\in[a,b],
\end{equation}
the right RLFD
\begin{equation}
\label{RLFD2}
{_xD_b^\alpha}f(x)=\frac{-1}{\Gamma(1-\alpha)}\frac{d}{dx}\int_x^b
(t-x)^{-\alpha}
f(t)dt=\left(-\frac{d}{dx}\right){_xI_b^{1-\alpha}}f(x),\quad
x\in[a,b],
\end{equation}
the left Caputo Fractional Derivative (CFD)
\begin{equation}
\label{CFD1}
{^C_aD_x^\alpha}f(x)=\frac{1}{\Gamma(1-\alpha)}\int_a^x
(x-t)^{-\alpha}\frac{d}{dt}f(t)dt={_aI_x^{1-\alpha}}\frac{d}{dx}f(x),
\quad x\in[a,b],
\end{equation}
the right CFD
\begin{equation}
\label{CFD2}
{_xD_b^\alpha}f(x)=\frac{-1}{\Gamma(1-\alpha)}\int_x^b
(t-x)^{-\alpha}
\frac{d}{dt}f(t)dt={_xI_b^{1-\alpha}}\left(-\frac{d}{dx}\right)f(x),
\quad x\in[a,b],
\end{equation}
where $\alpha$ is the order of the derivative.

The operators \eqref{RLFI1}--\eqref{CFD2} are obviously linear. We
now present the rule of fractional integration by parts for RLFI
(see for instance \cite{int:partsRef}). Let $0<\alpha<1$, $p\geq1$,
$q \geq 1$, and $1/p+1/q\leq1+\alpha$. If $g\in L_p([a,b])$ and
$f\in L_q([a,b])$, then
\begin{equation}
\label{ipi}
\int_{a}^{b} g(x){_aI_x^\alpha}f(x)dx
=\int_a^b f(x){_x I_b^\alpha} g(x)dx.
\end{equation}
In the discussion to follow, we will also need the following
formulae for fractional integrations by parts:
\begin{equation}
\label{ip}
\begin{split}
\int_{a}^{b}  g(x) \, {^C_aD_x^\alpha}f(x)dx &=\left.f(x){_x
I_b^{1-\alpha}} g(x)\right|^{x=b}_{x=a}+\int_a^b f(x){_x D_b^\alpha}
g(x)dx,\\
\int_{a}^{b}  g(x) \, {^C_xD_b^\alpha}f(x)dx &=\left.-f(x){_a
I_x^{1-\alpha}} g(x)\right|^{x=b}_{x=a}+\int_a^b f(x){_a D_x^\alpha}
g(x)dx.
\end{split}
\end{equation}

They can be derived using equations \eqref{RLFD1}--\eqref{CFD2},
the identity \eqref{ipi} and performing integration by parts.


\section{Main Results}
\label{main}

Let us consider the following problem:
\begin{equation}
\label{Funct1}
\begin{gathered}
\mathcal{J}(y)=\int_a^b L(x,y(x),{^C_a D_x^{\alpha}}
y(x),{^C_xD_b^\beta}y(x), y(a),y(b)) \, dx \longrightarrow \text{extr}\\
\quad (y(a)=y_{a}), \quad (y(b)=y_{b}).
\end{gathered}
\end{equation}
Using parentheses around the end-point conditions means that the
conditions may or may not be present. We assume that:
\begin{itemize}
\item[(i)]
$L(\cdot,\cdot,\cdot,\cdot,\cdot,\cdot) \in
C^1([a,b]\times\mathbb{R}^5; \mathbb{R})$;
\item[(ii)] $x\to \partial_3
L(x,y(x),{^C_a D_x^{\alpha}} y(x),{^C_xD_b^\beta}y(x), y(a),y(b))$
has continuous right RLFI of order $1-\alpha$ and right RLFD of
order $\alpha$, where $\alpha\in(0,1)$;
\item[(iii)] $x\to\partial_4 L(x,y(x),{^C_a D_x^{\alpha}}
y(x),{^C_xD_b^\beta}y(x), y(a),y(b))$ has continuous left RLFI of
order $1-\beta$ and left RLFD of order $\beta$, where
$\beta\in(0,1)$.
\end{itemize}
\begin{remark}
We are assuming that the admissible functions $y$ are such that
${^C_a D_x^{\alpha}} y(x)$ and ${^C_xD_b^\beta}y(x)$ exist on the
closed interval $[a,b]$.
\end{remark}
Along the work we denote by $\partial_iL$, $i=1,\ldots,6$, the
partial derivative of function
$L(\cdot,\cdot,\cdot,\cdot,\cdot,\cdot)$ with respect to its
$i$th argument.


\subsection{Necessary Optimality Conditions}
\label{sec3}

Next theorem gives necessary optimality conditions for the problem
\eqref{Funct1}.

\begin{theorem}
\label{Theo E-L1}
Let $y$ be a local extremizer to problem \eqref{Funct1}.
Then, $y$ satisfies the fractional Euler--Lagrange equation
\begin{multline}
\label{E-L1}
\partial_2L(x,y(x),{^C_a D_x^{\alpha}}
y(x),{^C_xD_b^\beta}y(x), y(a),y(b))+{_x D_b^{\alpha}}
\partial_3L(x,y(x),{^C_a D_x^{\alpha}} y(x),{^C_xD_b^\beta}y(x), y(a),y(b))\\
+ {_aD_x^\beta}\partial_4L(x,y(x),{^C_a D_x^{\alpha}}
y(x),{^C_xD_b^\beta}y(x), y(a),y(b)) = 0
\end{multline}
for all $x\in[a,b]$. Moreover, if $y(a)$ is not specified, then
\begin{multline}
\label{new:bca}
\int_a^b \partial_5L(x,y(x),{^C_a D_x^{\alpha}}
y(x),{^C_xD_b^\beta}y(x), y(a),y(b)) \, dx\\
= \left[{_x I_b^{1-\alpha}}
\partial_3L(x,y(x),{^C_a D_x^{\alpha}} y(x),{^C_xD_b^\beta}y(x), y(a),y(b))\right.\\
\left.\left.-{_aI_x^{1-\beta}}\partial_4L(x,y(x),{^C_a D_x^{\alpha}}
y(x),{^C_xD_b^\beta}y(x), y(a),y(b))\right]\right|_{x=a}
\end{multline}
if $y(b)$ is not specified, then
\begin{multline}
\label{new:bcb}
\int_a^b \partial_6L(x,y(x),{^C_a D_x^{\alpha}}
y(x),{^C_xD_b^\beta}y(x), y(a),y(b)) \, dx\\
= \left[{_aI_x^{1-\beta}}\partial_4L(x,y(x),{^C_a D_x^{\alpha}}
y(x),{^C_xD_b^\beta}y(x), y(a),y(b))\right.\\
\left.\left.-{_x I_b^{1-\alpha}}\partial_3L(x,y(x),{^C_a
D_x^{\alpha}} y(x),{^C_xD_b^\beta}y(x),
y(a),y(b))\right]\right|_{x=b}.
\end{multline}
\end{theorem}

\begin{proof}
Suppose that $y$ is an extremizer of $\mathcal{J}$. We can proceed
as Lagrange did, by considering the value of $\mathcal{J}$ at a
nearby function $\tilde{y}=y + \varepsilon h$, where $\varepsilon\in
\R$ is a small parameter, $h$ is an arbitrary admissible function.
We do not require $h(a)=0$ or $h(b)=0$ in case $y(a)$ or $y(b)$,
respectively, is free (it is possible that both are free). Let
\begin{equation*}
\varphi(\varepsilon) =\int_a^b L(x,y(x) + \varepsilon h(x),{^C_a
D_x^{\alpha}}(y(x) + \varepsilon h(x)),{^C_xD_b^\beta}(y(x) +
\varepsilon h(x)), y(a) + \varepsilon h(a),y(b) + \varepsilon h(b))
\, dx
\end{equation*}
Since ${^C_a D_x^{\alpha}}$ and ${^C_xD_b^\beta}$  are linear
operators, it follows that
\begin{equation*}
\begin{split}
{^C_a D_x^{\alpha}}(y(x) + \varepsilon h(x))&={^C_a
D_x^{\alpha}}y(x) + \varepsilon {^C_a D_x^{\alpha}}h(x)\\
{^C_xD_b^\beta}(y(x) + \varepsilon h(x))&={^C_xD_b^\beta}y(x) +
\varepsilon {^C_xD_b^\beta}h(x).
\end{split}
\end{equation*}
A necessary condition for $y$ to be an extremizer is given by
\begin{multline}
\label{eq:FT}
\left.\frac{d \varphi}{d\varepsilon}\right|_{\varepsilon=0} = 0\\
\Leftrightarrow \int_a^b \Bigl[
\partial_2L(\cdots)h(x)+\partial_3L(\cdots){^C_a D_x^{\alpha}}h(x)
+\partial_4L(\cdots){^C_xD_b^\beta}h(x)
+\partial_5L(\cdots)h(a)+\partial_6L(\cdots)h(b)\Bigl]dx\\
= 0 \, ,
\end{multline}
where $(\cdots) = \left(x,y(x),{^C_a D_x^{\alpha}}
y(x),{^C_xD_b^\beta}y(x), y(a),y(b)\right)$. Using formulae
\eqref{ip} for integration by parts, the second and the third
integral can be written as
\begin{equation}
\label{IP}
\begin{split}
\int_a^b\partial_3L(\cdots){^C_a
D_x^{\alpha}}h(x)dx&=\int_a^bh(x){_x
D_b^{\alpha}}\partial_3L(\cdots)dx+\left.{_x
I_b^{1-\alpha}}\partial_3L(\cdots)h(x)\right|^{x=b}_{x=a},\\
\int_a^b\partial_4L(\cdots){^C_xD_b^\beta}h(x)dx&=\int_a^bh(x){_a
D_x^{\beta}}\partial_4L(\cdots)dx-\left.{_a
I_x^{1-\beta}}\partial_4L(\cdots)h(x)\right|^{x=b}_{x=a}.
\end{split}
\end{equation}
Substituting \eqref{IP} into \eqref{eq:FT}, we get
\begin{multline}
\label{eq:aft:IP}
\int_a^b \Bigl[
\partial_2L(\cdots)+{_x
D_b^{\alpha}}\partial_3L(\cdots)+{_a
D_x^{\beta}}\partial_4L(\cdots)\Bigl]h(x)\\
+\left.{_x I_b^{1-\alpha}}\partial_3L(\cdots)h(x)\right|^{x=b}_{x=a}
-\left.{_a I_x^{1-\beta}}\partial_4L(\cdots)h(x)\right|^{x=b}_{x=a}\\
+\int_a^b \left(\partial_5L(\cdots)h(a)+\partial_6L(\cdots)h(b)\right)dx=0.
\end{multline}
We first consider functions $h(t)$ such that $h(a)=h(b)=0$.
Then, by the fundamental lemma of the calculus of variations,
we deduce that
\begin{equation}
\label{eq:EL} \partial_2L(\cdots)+{_x
D_b^{\alpha}}\partial_3L(\cdots)+{_a
D_x^{\beta}}\partial_4L(\cdots)=0
\end{equation}
for all $x\in[a,b]$. Therefore, in order for $y$ to be an extremizer
to the problem \eqref{Funct1}, $y$ must be a solution of the
fractional Euler--Lagrange equation. But if $y$ is a solution of
\eqref{eq:EL}, the first integral in expression \eqref{eq:aft:IP}
vanishes, and then the condition \eqref{eq:FT} takes the form
\begin{multline*}
h(b)\left\{\int_a^b\partial_6L(\cdots)dx -\Bigl[{_a
I_x^{1-\beta}}\partial_4L(\cdots)-{_x
I_b^{1-\alpha}}\partial_3L(\cdots)\Bigl]\left.\right|_{x=b}\right\}\\
+h(a)\left\{\int_a^b \partial_5L(\cdots)dx-\Bigl[{_x
I_b^{1-\alpha}}\partial_3L(\cdots)-{_a
I_x^{1-\beta}}\partial_4L(\cdots)\Bigl]\left.\right|_{x=a}\right\}=0.
\end{multline*}
If $y(a)=y_{a}$ and $y(b)=y_{b}$ are given in the formulation of
problem \eqref{Funct1}, then the latter equation is trivially
satisfied since $h(a)=h(b)=0$. When $y(a)$ is free, then
\begin{equation*}
\int_a^b \partial_5L(\cdots)dx-\Bigl[{_x
I_b^{1-\alpha}}\partial_3L(\cdots)-{_a
I_x^{1-\beta}}\partial_4L(\cdots)\Bigl]\left.\right|_{x=a}=0,
\end{equation*}
when $y(b)$ is free, then
\begin{equation*}
\int_a^b\partial_6L(\cdots)dx -\Bigl[{_a
I_x^{1-\beta}}\partial_4L(\cdots)-{_x
I_b^{1-\alpha}}\partial_3L(\cdots)\Bigl]\left.\right|_{x=b}=0
\end{equation*}
since $h(a)$ or $h(b)$ is, respectively, arbitrary.
\end{proof}

\begin{remark}
Conditions \eqref{E-L1}--\eqref{new:bcb} are only necessary for an
extremum. The question of sufficient conditions for an extremum is
considered in Subsection~\ref{sec4}.
\end{remark}

In the case $L$ does not depend on $y(a)$ and $y(b)$, by
Theorem~\ref{Theo E-L1} we obtain the following result.

\begin{corollary}[Theorem~1 of \cite{agrawalCap}]
If $y$ is a local extremizer to problem
\begin{equation*}
 \mathcal{J}(y)=\int_a^b L(x,y(x),{^C_a D_x^{\alpha}}
y(x),{^C_xD_b^\beta}y(x)) \, dx \longrightarrow \text{extr},
\end{equation*}
then $y$ satisfies the fractional Euler--Lagrange equation
\begin{multline*}
\partial_2L(x,y(x),{^C_a D_x^{\alpha}}
y(x),{^C_xD_b^\beta}y(x))+{_x D_b^{\alpha}}
\partial_3L(x,y(x),{^C_a D_x^{\alpha}} y(x),{^C_xD_b^\beta}y(x)) \\
+ {_aD_x^\beta}\partial_4L(x,y(x),{^C_a D_x^{\alpha}}
y(x),{^C_xD_b^\beta}y(x)) = 0
\end{multline*}
for all $x\in[a,b]$. Moreover, if $y(a)$ is not specified, then
\begin{equation*}
\left[{_x I_b^{1-\alpha}}
\partial_3L(x,y(x),{^C_a D_x^{\alpha}} y(x),{^C_xD_b^\beta}y(x))
\left.-{_aI_x^{1-\beta}}\partial_4L(x,y(x),{^C_a D_x^{\alpha}}
y(x),{^C_xD_b^\beta}y(x))\right]\right|_{x=a}=0,
\end{equation*}
if $y(b)$ is not specified, then
\begin{equation*}
\left[{_aI_x^{1-\beta}}\partial_4L(x,y(x),{^C_a D_x^{\alpha}}
y(x),{^C_xD_b^\beta}y(x))
\left.-{_x I_b^{1-\alpha}}\partial_3L(x,y(x),{^C_a
D_x^{\alpha}} y(x),{^C_xD_b^\beta}y(x))\right]\right|_{x=b}=0.
\end{equation*}
\end{corollary}

We note that the generalized Euler--Lagrange equation contains both
the left and the right fractional derivative. The generalized
natural conditions contain also the left and the right fractional
integral. Although the functional has been written only in terms of
the CFDs, necessary conditions \eqref{E-L1}--\eqref{new:bcb} contain
Caputo fractional derivatives, Riemann--Liouville fractional
derivatives and Riemann--Liouville fractional integrals.

Observe that if $\alpha$ goes to $1$, then the operators ${^C_a
D_x^{\alpha}}$ and ${_a D_x^{\alpha}}$ can be replaced with
$\frac{d}{dx}$ and the operators ${^C_xD_b^\alpha}$ and
${_xD_b^\alpha}$ can be replaced with $-\frac{d}{dx}$ (see
\cite{Podlubny}). Thus, if the ${^C_xD_b^\beta}y$ term is not present in
\eqref{Funct1}, then for $\alpha \rightarrow 1$
we obtain a corresponding result in the classical context
of the calculus of variations \cite{JP:DT:AZ}
(see also \cite[Corollary~1]{MalTor}).

\begin{corollary}
If $y$ is a local extremizer for
\begin{equation*}
\begin{gathered}
\mathcal{J}(y)=\int_a^b L(x,y(x),
y'(x), y(a),y(b)) \, dx \longrightarrow \text{extr}\\
\quad (y(a)=y_{a}), \quad (y(b)=y_{b}),
\end{gathered}
\end{equation*}
then
\begin{equation*}
\frac{d}{dx}\partial_{3}L(x,y(x), y'(x), y(a),y(b))=
\partial_2L(x,y(x), y'(x), y(a),y(b))
\end{equation*}
for all $x \in [a,b]$. Moreover, if $y(a)$ is free, then
\begin{equation*}
\partial_{3}L(a,y(a), y'(a), y(a),y(b))
=\int_{a}^{b}\partial_{5}L(x,y(x), y'(x), y(a),y(b))dx;
\end{equation*}
if $y(b)$ is free, then
\begin{equation*}
\partial_{3}L(b,y(b), y'(b), y(a),y(b))
=-\int_{a}^{b}\partial_{6}L(x,y(x), y'(x),
y(a),y(b))dx.
\end{equation*}
\end{corollary}


\subsection{Sufficient Conditions}
\label{sec4}

In this section we prove sufficient conditions that ensure the
existence of minimum (maximum). Similarly to what happens in the
classical calculus of variations, some conditions of convexity
(concavity) are in order.

\begin{definition}
Given a function $L$, we say that $L(\underline x,y,z,t,u,v)$ is
jointly convex (concave) in $(y,z,t,u,v)$, if $\partial_i L$ ,
$i=2,\ldots,6$, exist and are continuous and verify the following
condition:
\begin{multline*}
L(x,y+y_1,z+z_1,t+t_1,u+u_1,v+v_1)-L(x,y,z,t,u,v)\\
\geq (\leq) \partial_2 L(\bullet)y_1+\partial_3 L(\bullet)z_1
+ \partial_4 L(\bullet)t_1+\partial_5 L(\bullet)u_1
+\partial_6 L(\bullet)v_1
\end{multline*}
for all $(x,y,z,t,u,v)$, $(x,y+y_1,z+z_1,t+t_1,u+u_1,v+v_1)
\in [a,b]\times\mathbb R^5$, where $(\bullet)=(x,y,z,t,u,v)$.
\end{definition}

\begin{theorem}
\label{suff}
Let $L(\underline x,y,z,t,u,v)$ be jointly convex (concave) in
$(y,z,t,u,v)$. If $y_0$ satisfies conditions
\eqref{E-L1}--\eqref{new:bcb}, then $y_0$ is a global minimizer
(maximizer) to problem \eqref{Funct1}.
\end{theorem}

\begin{proof}
We shall give the proof for the convex case. Since $L$ is jointly
convex in $(y,z,t,u,v)$ for any admissible function $y_0+h$, we have
\begin{equation*}
\begin{split}
\mathcal{J}&(y_0+h)-\mathcal{J}(y_0)\\
&=\int_a^b \Bigl[L(x,y_0(x) + h(x),{^C_a D_x^{\alpha}}(y_0(x) +
h(x)),{^C_xD_b^\beta}(y_0(x) + h(x)),y_0(a) + h(a),y_0(b) + h(b))\Bigl.\\
&\qquad \qquad \Bigl.-L(x,y_0(x),{^C_a
D_x^{\alpha}}y_0(x),{^C_xD_b^\beta}y_0(x),y_0(a),y_0(b))\Bigl]dx\\
&\geq \int_a^b \Bigl[
\partial_2L(\star)h(x)+\partial_3L(\star){^C_a D_x^{\alpha}}h(x)
+\partial_4L(\star){^C_xD_b^\beta}h(x)+\partial_5L(\star)h(a)
+\partial_6L(\star)h(b)\Bigl]dx
\end{split}
\end{equation*}
where $(\star) = \left(x,y_0(x),{^C_a D_x^{\alpha}}
y_0(x),{^C_xD_b^\beta}y_0(x), y_0(a),y_0(b)\right)$. We can now
proceed analogously to the proof of Theorem~\ref{Theo E-L1}.
As the result we get
\begin{multline*}
\mathcal{J}(y_0+h)-\mathcal{J}(y_0)
\geq \int_a^b \Bigl[
\partial_2L(\star)+{_x
D_b^{\alpha}}\partial_3L(\star)+{_a
D_x^{\beta}}\partial_4L(\star)\Bigl]h(x)\\
+h(b)\left\{\int_a^b\partial_6L(\star)dx -\Bigl[{_a
I_x^{1-\beta}}\partial_4L(\star)-{_x
I_b^{1-\alpha}}\partial_3L(\star)\Bigl]\left.\right|_{x=b}\right\}\\
+h(a)\left\{\int_a^b \partial_5L(\star)dx-\Bigl[{_x
I_b^{1-\alpha}}\partial_3L(\star)-{_a
I_x^{1-\beta}}\partial_4L(\star)\Bigl]\left.\right|_{x=a}\right\}=0.
\end{multline*}
Since $y_0$ satisfy conditions \eqref{E-L1}--\eqref{new:bcb}, we
obtain $\mathcal{J}(y_0+h) - \mathcal{J}(y_0) \geq 0$.
\end{proof}


\subsection{Example}
\label{sec5}

We shall provide an example in order to illustrate our main results.

\begin{example}
Consider the following problem:
\begin{equation}
\label{EX} \mathcal{J}(y)=\frac{1}{2}\int_0^1 \left[({^C_0
D_x^{\alpha}}
y(x))^2+\gamma y^2(0)+\lambda(y(1)-1)^2\right]dx \longrightarrow \min\\
\end{equation}
where $\gamma,\lambda\in \R^+$. For this problem, the generalized
Euler--Lagrange equation and the natural boundary conditions (see
Theorem~\ref{Theo E-L1}) are given, respectively, as
\begin{equation}\label{Ex:el}
{_x D_1^{\alpha}}\left({^C_0 D_x^{\alpha}} y(x)\right)=0,
\end{equation}
\begin{equation}\label{Ex:ba}
\int_0^1\gamma y(0)dx=\left.{_x I_1^{1-\alpha}}({^C_0 D_x^{\alpha}}
y(x)\right)|_{x=0},
\end{equation}
\begin{equation}\label{Ex:bb}
\int_0^1\lambda (y(1)-1)dx=\left.-{_x I_1^{1-\alpha}}({^C_0
D_x^{\alpha}} y(x)\right)|_{x=1}.
\end{equation}
Note that it is difficult to solve the above fractional equations.
For $0<\alpha <1$ numerical method should be used. When $\alpha$
goes to $1$ problem \eqref{EX} tends to
\begin{equation}
\label{EX:1} \mathcal{J}(y)=\frac{1}{2}\int_0^1 \left[(
y'(x))^2+\gamma y^2(0)+\lambda(y(1)-1)^2\right]dx \longrightarrow \min\\
\end{equation}
 and equations \eqref{Ex:el}--\eqref{Ex:bb} could be replaced with
\begin{equation}\label{Ex:el:1}
y''(x)=0,
\end{equation}
\begin{equation}\label{Ex:ba:1}
\gamma y(0)=y'(0),
\end{equation}
\begin{equation}\label{Ex:bb:1}
\lambda (y(1)-1)=-y'(1).
\end{equation}
Solving equations \eqref{Ex:el:1}--\eqref{Ex:bb:1} we obtain that
\begin{equation*}
\bar{y}(x)=\frac{\gamma \lambda}{\gamma \lambda +\lambda
+\gamma}x+\frac{\lambda}{\gamma \lambda +\lambda +\gamma}
\end{equation*}
is a candidate for minimizer. Observe that problem \eqref{EX}
satisfies assumptions of Theorem~\ref{suff}. Therefore $\bar{y}$ is
a global minimizer to problem \eqref{EX:1}.
\end{example}


\section*{Acknowledgments}

The authors were partially supported by the R\&D unit CEOC, via FCT
and the EC fund FEDER/POCI 2010, and the research project
UTAustin/MAT/0057/2008. The first author was also supported
by Bia{\l}ystok University of Technology, via a project of the
Polish Ministry of Science and Higher Education ``Wsparcie
miedzynarodowej mobilnosci naukowcow''.

\smallskip

We are grateful to two anonymous reviewers for their comments.



\end{document}